\documentclass[10pt,a4paper,final,oneside,openright,reqno]{amsart} 
\usepackage{latexsym, amsfonts, amsthm, amsmath, amssymb, tikz, enumitem, booktabs, float,color,multicol,epsf, url,epsfig,epstopdf, array, setspace, graphicx, csquotes, xcolor,tikz-cd}
\usepackage[top=30mm, bottom=30mm, left=35mm , right=35mm]{geometry}

\usepackage{pinlabel}
\usepackage{amsfonts}
\usepackage{amsmath}
\usepackage{amssymb}
\usepackage[applemac]{inputenc}
\usepackage{url}
\usepackage{hyperref}
\usepackage{color}
\usepackage{lipsum}
\usepackage{mathtools}
\usepackage[T1]{fontenc}
\usepackage{tikz-cd}
\usepackage{url} 
\usepackage{hyperref} 
\usepackage{pdfpages}
\usepackage{graphicx}
\usepackage{ulem}

\newcommand\blfootnote[1]{%
  \begingroup
  \renewcommand\thefootnote{}\footnote{#1}%
  \addtocounter{footnote}{-1}%
  \endgroup
}
\newtheorem{theorem}{{Theorem}}[section]

\newtheorem{proposition}[theorem]{{Proposition}}

\newtheorem{lemma}[theorem]{{Lemma}}
\newtheorem{corollary}[theorem]{{Corollary}}
\newtheorem{fact}[theorem]{{Fact}}

\newtheorem{remark}[theorem]{{Remark}}

\newtheorem{question}[theorem]{{Question}}
\newtheorem{observation}[theorem]{{Observation}}

\newtheorem{problem}[theorem]{{Problem}}
\newtheorem{example}[theorem]{{Example}}

\definecolor{greenbf}{rgb}{0.3, 0.8 ,0.7}
\hypersetup{colorlinks=true, breaklinks=true, urlcolor= blue, linkcolor= red, citecolor= {blue},}

\def\R{\mathbb{R}}

\def\Z{\mathbb{Z}}
\def\S{\mathbb{S}}
\def\N{\mathbb{N}}

\def\F{\mathcal{F}}

\def\Isom{{\sf{Isom}}}
\def\Conf{{\sf{Conf}}}

\def\Aff{{\mathrm{Aff}}}
\def\Sim{\sf{Sim}}
\def\Aut{{\sf{Aut}}}

\def\Der{\sf{Der}}

\def\heis{{\mathfrak{heis}} }

\def\Aut{{\sf{Aut}}}

\def\Heis{{\sf{Heis}}}

\def\GL{{\mathrm{GL}}}
\def\O{{\sf{O}}}
\def\SO{{\sf{SO}}}

\def\SL{{\sf{SL}}}

\def\Mink{{\sf{Mink}} }

\def\Id{{\sf{Id}} }

\def\k{{\mathfrak{k}}}

\def\a{{\mathfrak{a}}}

\def\n{{\mathfrak{n}}}

\def\z{{\mathfrak{z}}}
\def\z{{\mathfrak{z}}}

\newcommand{\ad}{\mathrm{ad}}
\newcommand{\Ad}{\mathrm{Ad}}
\def\Heis{{\mathsf{Heis}}}

\usepackage{lipsum}

\usepackage{fancyhdr}
\begin{document}

 \title{Conformal quotients of plane waves, and Lichnerowicz conjecture in a locally homogeneous setting}
 
	\author [L. Mehidi]{Lilia Mehidi}
	\address{Departamento de Geometria y Topologia\hfill\break\indent
Facultad de Ciencias, Universidad de Granada, Spain\\}
	\email{lilia.mehidi@ugr.es
    \hfill\break\indent
	\url{https://mehidi.pages.math.cnrs.fr/siteweb/}}

\date{\today}
\maketitle   
\begin{abstract}
In the first part of the paper, we study conformal groups that act properly discontinuously and cocompactly on simply connected, non-flat homogeneous plane waves. We show that proper cocompact similarity actions that are not isometric can occur, in contrast to the behavior of Riemannian and Lorentzian affine similarity actions.
In the second part, we consider the Lorentzian conformal Lichnerowicz conjecture, which states that if the conformal group of a compact Lorentzian manifold acts without preserving any metric in the conformal class, then the manifold must be conformally flat. We prove the conjecture in a locally homogeneous setting.
\end{abstract}

\tableofcontents

\section{Introduction}

Two metrics $g$ and $g'$  on a manifold $M$  are conformally equivalent if $g'=e^f g$ for some smooth function $f$ on $M$. A conformal structure on $M$ is an equivalence class $[g]$ of a pseudo-Riemannian metric on $M$, making $(M,[g])$ a conformal manifold. The conformal group $\Conf(M,g)$ consists of diffeomorphisms preserving the conformal structure $[g]$. 
Within this group, transformations satisfying $\Phi^* g=cg$ for some constant $c>0$ are called \textbf{similarities}. 

\subsection{Properly discontinuous actions of similarity groups}
A similarity structure on a manifold $M$ consists of a maximal atlas whose charts are equipped with a pseudo-Riemannian metric, with transition maps given by similarity transformations. Of particular interest are similarity manifolds modeled on $(G,X)$-structures (in the sense of Ehresmann-Thurston), where $X$ is a smooth manifold with a faithful and transitive action of a finite-dimensional Lie group $G$. They are called $(G,X)$-manifolds. Here, $X$ is a pseudo-Riemannian space, and $G$ is the group of similarity transformations. If a subgroup $\Gamma \subset G$ acts properly discontinuously on $X$, then the quotient $\Gamma \backslash X$ is a manifold, called a quotient of $X$. These quotients are precisely the ``complete'' $(G,X)$-manifolds.

Remarkable similarity structures consist of  $(G,X)$-structures, where $X=\Mink^{p+q}$ is (the flat) pseudo-Riemannian Minkowski space
$\Mink^{p,q} =(\R^{p+q}, -dx^2_1 - \dots - dx^2_{p} + dy_1^2 + \dots + dy^2_{q})$,
and $G$ is its similarity group.   For $p=0$, the similarity group of the Euclidean space $\R^q$  is the Euclidean group  extended by homotheties $\Sim(\R^q)=(\R \times \O(q)) \ltimes \R^q$. Similarly, the similarity group of the pseudo-Riemannian Minkowski space is $\Sim(\Mink^{p,q})=(\R \times \O(p,q)) \ltimes \R^{p+q}$. These are special affine structures.
The study of compact affine manifolds, i.e. modeled on $(\Aff(\R^n),\R^n)$, has a rich history, with many open questions and conjectures. The case of flat similarity structures represents an important subcase.

The first examples of similarity manifolds modeled on $(\Sim(\R^n), \R^n)$ are Euclidean space forms, namely quotients of $\R^n$ by a discrete subgroup of (affine) isometries of $\R^n$. It turns out that those quotients correspond exactly to the complete similarity structures: a subgroup of $\Sim(\R^n)$ acting properly discontinuously on $\R^n$ is necessarily isometric. Kamishima \cite{kamishima_similarity} extended this  to compact Lorentzian similarity manifolds modeled on $(\Sim(\Mink^{1,n}), \Mink^{1,n})$. It is known from \cite{goldman_fundamental_group} that the fundamental group of a compact flat Lorentzian manifold is virtually solvable. Using similar techniques, Kamishima showed that the fundamental group of a compact Lorentzian similarity manifold is also virtually solvable. Using this, he  obtains that any subgroup of $\Sim(\Mink^{1,n})$ acting properly discontinuously and cocompactly on $\Mink^{1,n}$ must be isometric. 

\subsection{The case of homogeneous plane waves}
Now, we change the model and consider a non-flat $1$-connected homogeneous plane wave.  
Such a space is a Lorentzian manifold with a lightlike parallel vector field $V$ which can be seen as a deformation as well as a generalization of the Lorentzian Minkowski space. Since every $1$-connected homogeneous plane wave is conformal to a complete one, we take as a model space $(X,V)$, a non-flat $1$-connected complete homogeneous plane wave. 
It is shown in \cite[Section 6]{article1} that it has global coordinates, known as Brinkmann coordinates, in which the metric, in dimension $n+2$, takes the form
\begin{align}\label{Introduction: Eq Brinkmann coordinates}
    (\R^{n+2},\, 2 du dv + S_{ij}(u) x^i x^j du^2 + \sum_{i=1}^{n} (dx^i)^2),
\end{align}
where $S(u)=(S_{ij}(u))$ is a symmetric matrix with particular forms. When $S=0$, this corresponds to the usual Lorentzian scalar product of Minkowski space.   

Plane waves are interesting objects from the point of view of conformal geometry, as they are non-compact Lorentzian manifolds with an essential conformal group. They become even more interesting after the following dichotomy, shown recently in \cite{Alekseevsky2024}:  a conformally homogeneous $1$-connected manifold whose conformal group is essential is either conformally flat, or conformal to a (complete) homogeneous plane wave (see next paragraph, where these aspects are discussed). \medskip

The Killing fields of homogeneous plane waves are studied in \cite{Blau}, and the connected component of the isometry group is explicitly given in \cite{article1}, where  further global aspects of these spaces are considered.
In \cite{conformal_Killing_PW}, it was shown   that the Lie algebra of conformal Killing  fields of a homogeneous plane wave is a one-dimensional extension of the Lie algebra of Killing fields, consisting of homothetic vector fields. This result was extended in \cite{Alekseevsky2024} to  conformal groups of $1$-connected homogeneous plane waves, showing that they act by similarities.
Consequently, any compact quotient $M:=\Gamma \backslash X$ inherits a similarity structure. \medskip

The global metrics on $\R^{n+2}$ of the previous form, with $S$  constant with respect to $u$ and non-degenerate, define Cahen-Wallach spaces of dimension $n+2$. These spaces are indecomposable (globally) symmetric plane waves, introduced by Cahen and Wallach \cite{cahen1970lorentzian}. Leistner and Teisseire investigated the conformal compact quotients of Cahen-Wallach spaces of  ``imaginary type'', i.e. for which $S$ in (\ref{Introduction: Eq Brinkmann coordinates}) is negative definite. 
They proved that the groups acting properly discontinuously and cocompactly on these spaces are isometric, aligning with the previously mentioned results on similarity (affine) manifolds. 
The authors asked  whether such actions could exist without being isometric in the general Cahen-Wallach case. Motivated by this, and adopting a unifying approach, we investigate the subgroups of the conformal groups of $1$-connected (non-flat) homogeneous  plane waves acting properly discontinuously and cocompactly.   

It turns out that the more general setting of plane waves allows for strikingly different behavior. We show the following
\begin{theorem}
The groups acting properly discontinuously and cocompactly on $1$-connected complete homogeneous plane waves are either isometric or an extension of an isometric subgroup by a cyclic group generated by a (homothetic) similarity transformation.    
\end{theorem}
We also construct examples of such groups that are not contained in the isometry group, highlighting the richer and more flexible structure of proper similarity group actions in this context.  

\subsection{On the Lichnerowicz conjecture}
The Riemannian sphere $\S^n$  is the only compact Riemannian manifold with a non-compact conformal group, as conjectured by Lichnerowicz in 1964 and later confirmed by M. Obata and J. Ferrand \cite{obata, ferrand}. 
In the Riemannian setting, a non-compact conformal group means that the group does not preserve a Riemannian metric, as the isometry group of any compact Riemannian manifold must be compact. We say that the group action is ``essential''. 

The concept of essentiality extends to any Lie group $G$ that preserves some geometric structure: the action of $G$ is said to be ``non-essential'' if it preserves a ``stronger'' geometric structure. 
For a pseudo-Riemannian manifold $(M,g)$, the conformal group is said to be essential if it does not preserve any pseudo-Riemannian metric on $M$. In fact, if the conformal group preserves a volume form, then it preserves a metric in the conformal class. Hence, essentiality means that $\Conf(M,[g])$ does not preserve any metric in the conformal class.
However, unlike in the Riemannian case, essentiality in non-Riemannian signatures does not amount to the non-compactness of the conformal group. 
 
An analogous question was formulated  in the pseudo-Riemannian setting \cite{DG},  under the name of the pseudo-Riemannian Lichnerowicz Conjecture. It is asked whether a compact pseudo-Riemannian manifold  $(M,g)$ of dimension $n \geq 3$ with an essential conformal group is conformally flat.

However, the conjecture fails in pseudo-Riemannian settings with index greater than $1$. Indeed, Frances \cite{frances_pseudo-Lichnerowicz} constructed a non-conformally flat metric on $\S^1 \times \S^n$ of index $(p,q)$, for any $p, q \geq 2$, whose conformal group is essential. It turns out that these examples are quotients of higher-index symmetric plane waves; in particular, they are locally homogeneous.
In the (conformally) homogeneous setting, a recent work by Belraouti, Deffaf, Raffed, and Zeghib \cite{belraouti} shows that the conjecture holds, under certain assumptions on the structure of the conformal group.
The Lorentzian case remains an open problem, with many works and recent developments. Melnick and Pecastaing \cite{melnick_pecastaing} proved that the conformal group of a compact, $1$-connected, real-analytic Lorentzian manifold is compact, hence non-essential. In dimension $3$, Frances and Melnick \cite{frances_melnick} proved the conjecture for real-analytic Lorentzian manifolds without the $1$-connectedness assumption. 

\subsubsection{\textbf{Lichnerowicz conjecture in the case of $(G,X)$-manifolds}} 
While a complete resolution of the Lorentzian Lichnerowicz conjecture seems challenging, a promising case arises when the manifold is locally homogeneous. This is particularly interesting given that Frances' counterexamples, mentioned earlier, for index greater than $1$ fall within this class.

Our attempt is to consider this question in the context of compact $(G,X)$-manifolds. Let $X$ be a homogeneous manifold of a Lie group $G$, such that $G$ acts by preserving a conformal Lorentzian structure on $X$. A $(G,X)$-manifold $M$ inherits a conformal structure locally isomorphic to that of $X$. It is in particular locally homogeneous. If we assume that the conformal group of this structure is essential on $M$, then $G$ must be essential on $X$. We ask the following

\begin{question}\label{Question (G,X)-essentiality}
Let $X$ be a homogeneous manifold of a Lie group $G$, such that $G$ acts by preserving a conformal Lorentzian structure on $X$. 
Is a $(G,X)$-compact manifold $M$ essential if and only if $X$ is conformally flat ?    
\end{question}

\subsubsection{\textbf{Reduction to the case of homogeneous plane waves}}
Homogeneous plane waves play a fundamental role in this context. In a recent work, Alekseevsky and Galaev \cite{Alekseevsky2024}  proved that a $1$-connected conformally homogeneous manifold with an essential conformal group is either conformally flat or conformal to a (complete) homogeneous plane wave. Thus, in the study of Question \ref{Question (G,X)-essentiality}, it is sufficient to consider $X$ as a $1$-connected (non-flat) homogeneous  plane wave, with $G$ its conformal group. Henceforth, we assume that $X$ is a $1$-connected (non-flat) homogeneous  plane wave.  This paper examines the question in the case of complete $(G,X)$-models, i.e. when $M$ is a conformal compact quotient of $X$. This leads to the study of groups acting properly discontinuously and cocompactly on $X$.%\newpage

We obtain the following
\begin{theorem}
Let $M$ be a conformal compact quotient of a $1$-connected non-flat homogeneous  plane wave. Then, the action of the conformal group of $M$ is non-essential.     
\end{theorem}

\begin{corollary}\label{Corollary: Non essentiality in the homog case}
Let $X=G/H$ be a conformally homogeneous Lorentzian space. The conformal group of a (conformal) compact quotient of $X$ is essential if and only if $X$ is conformally flat. 
\end{corollary}

\subsubsection{\textbf{More general compact models: a Fried-type theorem ?}}
For a non-flat $1$-connected homogeneous plane wave $X$, we have shown that complete compact models are non-essential.  However, incomplete models also exist. In the context of compact similarity Riemannian manifolds, Hopf manifolds serve as examples of such spaces. These manifolds are quotients of $\R^n \smallsetminus \{0\}$ by a cyclic group generated by a contraction of $\R^n$, and they are homeomorphic to $\S^{n-1} \times \S^1$. Fried's theorem \cite{fried} asserts that Hopf manifolds are finite covers of the incomplete similarity structures, so that these manifolds and the Euclidean space forms are the only compact similarity manifolds. 
In the case of plane waves,  we formulate the following problem as part of Question \ref{Question (G,X)-essentiality}: 
\begin{problem}\label{Question 2}
Let $X$ be a $1$-connected non-flat homogeneous   plane wave, and $G$ its conformal group. 
\begin{enumerate}
    \item What are the incomplete compact models of this $(G,X)$-structure ? Are they Kleinian ? Do they satisfy a Fried-type theorem ?
    \item Study essentiality of these models. 
\end{enumerate}
\end{problem}
The study of Problem \ref{Question 2} will be the subject of an upcoming work.\bigskip

\noindent \textbf{Further ?} Finally, let us mention that  compact (conformally) locally homogeneous Lorentzian manifolds are more general than compact $(G,X)$-manifolds, where $(G,X)$ is a Lorentzian conformal structure.  Indeed, given such a manifold $M$, the universal cover $\widetilde{M}$ is not necessarily homogeneous, for its Killing fields are not necessarily complete. 
However, as in the homogeneous case, the cases where the conformal group of $\widetilde{M}$ fails to preserve a conformal metric are special and deserve to be understood.

\subsection*{Organization of the paper}
In Section 2, we recall the necessary background on homogeneous plane waves, their isometry groups, and conformal groups. Section 3 is devoted to the study of properly discontinuous and cocompact group actions of the conformal group, where we construct examples of groups acting properly discontinuously and cocompactly, but not via isometries. Section 4 focuses on the non-essentiality of the conformal groups of compact quotients.

\section{Preliminaries}
We refer the reader to \cite{article1} for the precise definition of a plane wave. Here, in the homogeneous setting, we recall only its characterization as a homogeneous space.

\subsection{Isometry group} The $(2n+1)$-dimensional Heisenberg group $\Heis_{2n+1}=\R^n \ltimes \R^{n+1}$ is the subgroup of $\Aff(\R^{n+1})$ defined by
$$\Heis_{2n+1}=\left\{\begin{pmatrix}
    1 &\alpha^{\top} \\
    0 &I_n
\end{pmatrix} | \  \alpha \in \R^{n} \right\}\ltimes \R^{n+1}.$$
Denote by $A^+=\R^n$ the abelian subgroup of unipotent matrices,  by $A^-$ the subgroup $\{0\} \times \R^n$ of the translation part, and by $Z$ the translation subgroup $\R \times \{0\}$, which is the center of the Heisenberg group. Their Lie algebras are denoted by $\a^+, \a^-,$ and $\z$, respectively. Set $\a:=\a^+\oplus \a^-$. \medskip

\noindent \textbf{Identity component.} Let $(X,V)$ be a $1$-connected non-flat  homogeneous plane wave of dimension $n+2$.
In \cite{article1}, it is shown that the identity component of the isometry group has finite index in the full isometry group. Moreover, the identity component of the isometry group is computed and has the form
$$G_\rho:=(\R \times K) \ltimes_{\rho} \Heis_{2n+1},$$
where 
\begin{itemize}
    \item $K$ is a closed subgroup of $\SO(n)$
    \item $\rho$ is a morphism $\rho: \R\times K \to \Aut(\Heis_{2n+1})$, where $\rho$ restricted to $\R$ is given by $\rho(t)= e^{t{\bf L}}$,  with ${\bf L} \in \Der(\heis_{2n+1})$ preserving the decomposition $\heis_{2n+1}=\a \oplus \z$,
    \item $\rho(k)$, for  $k\in K$, acts as the identity on the center of $\Heis_{2n+1}$,  and coincides with the standard action of $k$ on $A^+$ and $A^-$. 
\end{itemize}
Then, $X$ is identified with the quotient $X_{\rho }:=G_\rho/I$, where  $I= K\ltimes A^+$. The $\rho$-actions for which $G_\rho$ preserves a Lorentzian metric on $G_\rho/I$ are characterized in \cite{article1}; in this case, the metric is necessarily a plane wave. The proofs here work without this restriction on the $\rho$-action. \medskip

Henceforth, we will write $\R_{\bf L}$ instead of $\R$ to specify that $\R$ acts on $\Heis_{2n+1}$ via the one parameter subgroup $e^{t {\bf L}} \subset \Aut(\Heis_{2n+1})$. 

\begin{fact}$($\cite[Fact 3.1]{article2_crelle}$)$.\label{Fact: V defined by Z}
 The parallel vector field $V$ is generated by the action of the center of $\Heis_{2n+1}$.   
\end{fact}

\noindent\textbf{The full isometry group.} The full isometry group is not written in \cite{article1}; however, it can be directly deduced from the results there. 

The orthogonal distribution $V^{\perp}$ of a plane wave is parallel, hence integrable, defining a codimension-one foliation  $\mathcal{F}$. 
This foliation  is defined by the non-singular closed $1$-form $\omega=g(V,.)$, so it is transversely affine. 
Since $X$ is $1$-connected and homogeneous, the leaf space $\xi:=X/\mathcal{F}$ is diffeomorphic to $\R$, with a global affine parameter given by a section of a submersion $u: X\to \R$ satisfying $\omega = du$. 

Any isometry preserving $\R V$ induces an action on $\xi$, preserving its affine structure. As noted in the beginning of \cite[Section 5]{article1}, a non-flat plane wave admits a unique parallel null vector field $V$ (up to scale), so any isometry of $X$ maps $V$ to $\alpha V$, for some $\alpha \in \R \smallsetminus \{0\}$.   This yields a representation $\pi: \Isom(X) \to \Aff(\R)$, whose kernel  consists of isometries acting trivially on  $\xi$. These are given by $ \widehat{K} \ltimes_{\rho} \Heis_{2n+1}$,
where $\widehat{K}$ is a closed subgroup of $\O(n)$ and a finite extension of $K$. Isometries preserving $V$ but not in the kernel of $\pi$ correspond to the translation subgroup $\R$ in $\Aff(\R)$, i.e. they preserve a translation structure on $\xi$.  By Fact \ref{Fact: V defined by Z}, these are given by $ (\R_{\bf L} \times \widehat{K}) \ltimes_{\rho} \Heis_{2n+1}$ if the $\R_{\bf L}$-action on $\Heis_{2n+1}$ is trivial on its center, and by $\widehat{K} \ltimes_{\rho} \Heis_{2n+1}$ otherwise.

Since the full isometry group has finite index in $G_\rho$ (\cite[Proposition 5.1]{article1}), any isometry which is not in $G_\rho$ maps $V$ to $a V$, with $\alpha^2=1$. The case $\alpha=1$ is  discussed above,  so consider an isometry $\sigma$ that sends $V$ to $-V$. By the invariance of the affine structure on $\xi$, we have $u \circ \sigma = -u +b$, where $u \in C^\infty(X,\R)$ is a global affine parameter. Moreover, $\sigma^2$ is an isometry that acts trivially on the leaf space $\xi$. Up to composing with its inverse, we can assume $\sigma^2= \Id$. 
Therefore, in  the global Brinkmann coordinates (\ref{Introduction: Eq Brinkmann coordinates}), this extra isometry is given by $$\sigma(v,x,u)=(-v,x,-u+b),$$
up to composition with a reflection of $X$ acting trivially on $\xi$. 
We see that $\sigma$ exists if and only if there exists some $b\in \R$ such that $S(-u+b)=S(u)$ in the metric (\ref{Introduction: Eq Brinkmann coordinates}). In particular, such spaces are complete.  In this case, the full isometry group is
$$\widehat{G}_\rho:=(E_{\bf L}(1) \times \widehat{K}) \ltimes_{\rho} \Heis_{2n+1},$$ 
where $E_{\bf L}(1):= \langle \sigma \rangle \ltimes \R_{\bf L}$ is isomorphic to the Euclidean group of $\R$, and $\widehat{K}$ is a closed subgroup of $\O(n)$ and a finite extension of $K$. 
Cahen-Wallach spaces, where  $S(u)$ is constant, always admit this extra isometry. \medskip

In the generic case, the full isometry group is $$\widehat{G}_\rho:=(\R_{\bf L} \times \widehat{K}) \ltimes_{\rho} \Heis_{2n+1}.$$

\subsection{Conformal group}
The conformal group of $X$ is computed in  \cite{Alekseevsky2024}: \begin{theorem}\cite[Theorem 2]{Alekseevsky2024}
Let $(X,g)$ be a $1$-connected homogeneous non-conformally flat plane wave. Then the conformal group $\Conf(X,g)$ consists of similarities and it is a $1$-dimensional extension of the group of isometries.
\end{theorem}

The full conformal group is given by $$\widehat{G}:= (\R_{\bf H} \times E_{\bf L}(1) \times \widehat{K}) \ltimes_{\rho'} \Heis_{2n+1},$$ 
(or an index $2$ subgroup of it), where
\begin{itemize}
    \item $\rho'$ coincides with $\rho$ on $E_{\bf L}(1) \times \widehat{K}$,  
    \item $\rho'$ restricted to $\R_{\bf H}$ is given by $\rho'(t)=e^{t \bf H}$, with $${\bf H}:=
    \begin{pmatrix}
        I_n & 0 & 0\\
        0 & I_n & 0\\
        0 & 0 & 2
    \end{pmatrix} \in \Der(\heis_{2n+1}),
    $$
    written in the decomposition $\heis= \a^+ \oplus \a^- \oplus \z$.   The associated one-parameter group $e^{t \bf H}$ acts as homotheties on $\heis_{2n+1}$.\medskip
\end{itemize}

\noindent Then, $X$ is identified with the quotient $X_{\rho '}:=\widehat{G}/I'$, where  $I'= (\R_{\bf H} \times \langle \sigma \rangle \times \widehat{K})\ltimes A^+$. 
\bigskip

The identity component of the conformal group  is  given by $$G:=(\R_{\bf H}\times \R_{\bf L} \times K) \ltimes \Heis_{2n+1}.$$ 
It has finite index in $\widehat{G}$. So if $\Gamma$ is a discrete subgroup of $\widehat{G}$, then $\Gamma \cap G$  has finite index in $\Gamma$. 

\begin{observation} $\Gamma \cap G$ is contained in the identity component $G_\rho$ of the isometry group if and only if $\Gamma$ is contained in the isometry group $\widehat{G}_\rho$.     
\end{observation}

In the study of  quotients of $X$ by subgroups of $\widehat{G}$ acting properly discontinuously and cocompactly, we consider the quotients up to finite covering. By the observation above, taking the finite-index subgroup $\Gamma \cap G$ does not affect the property of $\Gamma$ being, or not being, contained in the isometry group.
Thus, we will always assume that $\Gamma$ is contained in the identity component $G$ of the conformal group.

\begin{fact}
The conformal group $\widehat{G}$ preserves the line field $\R V$, hence induces an action on the leaf space $\xi$ which preserves the affine structure on it.  And the subgroup $G$ preserves a translation structure. 
\end{fact}
\begin{proof}
This follows from the fact that any conformal transformation of $X$ sends $V$ to $\alpha V$, with $\alpha \in \R \smallsetminus \{0\}$, and acts on $g$ by similarities, hence   sends the $1$-form $\omega$ defining $\F$ to   $\lambda \omega$, with $\lambda \in \R \smallsetminus \{0\}$. 
\end{proof}

\textbf{Notation:} 
 We denote by $p: G \to \R_{\bf H} \times \R_{\bf L}$ and $r: G \to \R_{\bf H} \times \R_{\bf L} \times K$ the projection morphisms. We also introduce the following projection morphisms: $p_{\bf H}: G \to \R_{\bf H}$, $p_{\bf L}: G \to \R_{\bf L}$, and $p_K: G \to K$. Henceforth, we write $\Heis$ instead of $\Heis_{2n+1}$. 

Let $g \in \widehat{G}$. Then $g$ has a unique representation $g=(a,x)$, where $a \in \R_{\bf H} \times \R_{\bf L} \times K \subset \Aut(\Heis)$ and $x \in \Heis$. When $a=1$ in this representation, we  write $x$ instead of $(1,x)$. If $a(x)$ denotes the image of $x$ under $a$, then 
    $$(a_1,x_1)(a_2,x_2)=(a_1 a_2,\, x_1 \cdot a_1(x_2)),$$
where $\cdot$ denotes the multiplication in $\Heis$.

\section{Conformal compact quotients of $1$-connected homogeneous plane waves}

In \cite{Blau, article1}, non-flat $1$-connected homogeneous plane waves fall into two distinct families. Using the notation introduced above, the difference between these families lies in the action of ${\bf L}$ on the center of $\heis$: in one case, ${\bf L}$ acts trivially, while in the other, it acts non-trivially. 
The homogeneous plane waves of the first family are complete, whereas those of the second family are incomplete.\medskip

\textbf{Reduction:} Let $(X,V)$ be a $1$-connected non-flat homogeneous plane wave. As observed in \cite{conformal_Killing_PW} (see also \cite{Alekseevsky2024}), $X$ is conformal to a complete homogeneous plane wave. Therefore, when studying the essentiality of conformal groups of compact quotients, we may assume the model is complete. Henceforth, we assume that $X$ is complete, i.e., the action of ${\bf L}$ on the center of $\heis$ is trivial.

\begin{theorem}\label{Theorem 1}
A subgroup $\Gamma \subset G$ acting cocompactly and properly discontinuously on $X$ is 
\begin{itemize}
    \item either contained in the isometry group, 
    \item or it has a non-trivial projection on $\R_{\bf H}$, in which case its projection to $\R_{\bf L} \times \R_{\bf H}$ is isomorphic to $\Z$. In this case, up to finite index, $\Gamma \simeq \langle \hat{\gamma} \rangle \ltimes \Gamma_0$, where $\Gamma_0 \subset \Heis$ is abelian of rank $n+1$. 
\end{itemize}  
\end{theorem}
\subsection{Introductory lemmas} 
We begin by stating a few introductory (easy) lemmas that will be used many times throughout the proofs. \\

Let $f=(a,x) \in \Aff(\R^m)$, where $a \in \GL_m(\R)$  is its linear part and $x$ is the translation part. It is easy to see that if $a$ does not have $1$ as an eigenvalue, then $f$ has a unique fixed point. Let $y \in \R^m$ be this fixed point. Then, conjugation by the translation in $y$ results in an affine transformation that fixes $0$. 
This is, in fact, also true when replacing $\R^m$ with $N$, where $N$ is a  $1$-connected $2$-step nilpotent group. We provide a proof of this in the next lemma.

\begin{lemma}\label{Lemma fixed point: (A,b)-->(A,0)}
Let $(a,x) \in \Aut(N) \ltimes N$ such that $a$ has no eigenvalue  equal to $1$. Then, there exists $x_1 \in N$ such that $x_1 (a,x) x^{-1}_1=(a,1)$. 
\end{lemma}
\begin{proof}
We have $$x_1 (a,x)\, x^{-1}_1=(a,x_1 x\, a(x_1^{-1})).$$ In the abelian case, the problem amounts to finding $x_1$ such that $(a-\Id)(x_1)=x$. This equation admits a solution since the assumption that  $a$ has no eigenvalue equal to $1$ ensures that $a-\Id$ is invertible. 
Now, consider the projection morphism $p: N \to N/C(N)$, where $C(N)$ is the center of $N$. Denote by $\overline{x}:=p(x)$ for $x \in N$, and by $\overline{a}$ the induced action  of $a \in \Aut(N)$ on $N/C(N)$.  Projecting the equation above onto $N/C(N)$, which is abelian, we obtain the existence of $\overline{x_1} \in N/C(N)$ such that  $\overline{x_1}\, \overline{x}\,\overline{a}(\overline{x_1}^{-1})=\overline{0}$. 
Then, we have $x_1 x\, a(x_1^{-1})=z$, for some $z \in C(N)$.  There exists $z_1 \in C(N)$ such that $(\Id-a)(z_1)=z^{-1}$.
Setting $x_2:=x_1 z_1$, we compute  $x_2 x\, a(x_2^{-1})= x_1 x \,a(x_1^{-1}) z_1 a(z_1^{-1})=1$. Thus, $x_2$ satisfies the required condition. 
\end{proof}

\begin{corollary}\label{Corollary 1}
Let $g \in G$  such that $p_{\bf L}(g)=1$ and $p_{\bf H}(g) \neq 1$, then $g$ has a fixed point.
\end{corollary} 
\begin{proof}
Since the real eigenvalues of $p_K(g)$ are equal to $1$, and the eigenvalues of $p_{\bf H}(g)$ are all real and different from $1$,  Lemma \ref{Lemma fixed point: (A,b)-->(A,0)} ensures the existence of $x \in \Heis$ such that $x g\, x^{-1}=(a,1)$, where $a=r(g) \in \R_{\bf H} \times K$. Since $a$ has a fixed point in $X$, $g$ also has a fixed point in $X$. 
\end{proof}

\begin{lemma}\label{Lemma: eigenvalues > 1 implies homothetic}
Let $a \in \Aff(\R^m)$ be such that all its eigenvalues have norm $<1$. Then, the action of $a$ on $\R^m$ is a contraction, i.e. for any $x \in \R^m$, the sequence of forward iterations $(a^k(x))_{k \in \N}$ converges to $0$. 
\end{lemma}
\begin{proof}
Let $J(\lambda)$ be a Jordan bloc of $a$ of dimension $d$, where $\lambda$ is the corresponding (complex) eigenvalue. The non-zero entries of $(J(\lambda))^k$ are of the form $$\lambda^{k-j} P_j(k), \text{\;\;\,for\;} j \in \{0,\dots,d-1\},$$ where  $P_j(k)$ are polynomials. As $k \to +\infty$, the norms of these entries tend to zero. 
\end{proof}

\begin{corollary}\label{Corollary: sequence converging to 1}
Let $\Lambda$ be a subgroup of $G$. Suppose there exists  $(a,x) \in \Lambda$ such that $\Ad_a \in \Aut(\heis)$ has all its eigenvalues with norm $<1$. If $\Lambda$ is discrete, then $\Lambda \cap (K \ltimes \Heis) = \{1\}$.    
\end{corollary}
\begin{proof}
Write $g=(a,x)$, and let $h=(a_1,x_1) \in \Lambda \cap (K \ltimes \Heis)$. By Lemma \ref{Lemma fixed point: (A,b)-->(A,0)}, we may assume, after conjugation by an element of $\Heis$, that $x=1$.    
Then $g^k h g^{-k} = (a^k a_1 a^{-k}, a^k(x_1))$. Since $a^k a_1 a^{-k}$ is a sequence in $K$, which is compact, it has an accumulation point. Moreover, by Lemma \ref{Lemma: eigenvalues > 1 implies homothetic}, the sequence $a^k(x_1)$ converges to $1$. Since $\Lambda$ is discrete, the set $\{a^k(x_1), k \in \N\}$ must be finite, which forces $x_1=1$.   
\end{proof}

\begin{lemma}\label{Lemma: Gamma does not intersect the center}
Let $\Lambda$   be a subgroup of $G$ which is not contained in the isometry group $G_\rho$. If $\Lambda$ is discrete, then   $\Lambda \cap Z = \{1\}$. 
\end{lemma}
\begin{proof} 
Let $g \in \Lambda$ have a non-trivial projection on $\R_{\bf H}$, and let $z \in \Lambda \cap Z$. 
We have $g^k z g^{-k}=a^k(z)$, where $a=r(g)$. Since the actions of  $\R_{\bf L}$ and $K$ on the center of $\Heis$ are trivial, and the action of $\R_{\bf H}$ is contracting (up to forward or backward iteration), it follows that the sequence  $(a^k(z))_{k \in \Z}$ has an accumulation point at $1$. As $\Lambda$ is discrete, the set $\{a^k(z), k \in \N\} \subset \Lambda$ must be finite, which forces $z=1$.       
\end{proof}

\begin{lemma}$($\cite[Lemma 5.3]{article1}$)$\label{Lemma product R x K}
Let $Q$ be the semi-direct product $Q=\R\ltimes C$ where $C$ is a compact connected Lie group. Then $Q$ is isomorphic to the product $\R\times C$. \medskip
\end{lemma}

The last lemma we will need in this section is Lemma \ref{Lemma appendix} below. It is similar to \cite[Proposition 5.4]{article2_crelle}, but we present it here in a form that allows for systematic application. The proof however is the same, so to avoid unnecessary lengthening of the paper, we will not rewrite the proof. 

Let $Q \ltimes C \subset \Aut(\Heis)$ preserving the decomposition $\heis=\a \oplus \z$, and such that $C$ is compact. Define $$G :=(Q\ltimes C)\ltimes N.$$ Let $\Gamma$ be a discrete subgroup of $G$.
Under some conditions, the proposition below ensures the existence of a nilpotent syndetic hull of $\Gamma$, i.e. a connected nilpotent subgroup of $G$ containing $\Gamma$ as a lattice.

\begin{lemma}\label{Lemma appendix}
Let $\Gamma$ be a discrete subgroup of $G$, whose projection on $Q$ is dense. Then, up to finite index, $\Gamma$ is a lattice in a connected closed nilpotent subgroup of $G$.    
\end{lemma}

\subsection{Proof of Theorem \ref{Theorem 1}} 
We will examine all possible closures of the projection of 
$\Gamma$ to $\R_{\bf H} \times \R_{\bf L} \simeq \R^2$.

\begin{proposition}\label{Proposition: p(Gamma)=RxR}
If $\overline{p(\Gamma)} = \R_{\bf H} \times \R_{\bf L}$, then $\Gamma$ does not act  properly discontinuously and cocompactly on $X$.    
\end{proposition}
\begin{proof}
By Lemma \ref{Lemma appendix},  $\Gamma$ is a cocompact lattice in a connected nilpotent subgroup $N$ of $G$. 
We will show that $N \cap \Heis \neq \{1\}$. 
Our claim is that $\dim N \geq n+2$. Indeed, $\Gamma$ acts properly and cocompactly on the $K(\pi,1)$ space $N/C \simeq \R^k$, where $C$ is the maximal compact subgroup of $N$. So its cohomological dimension is equal to $\dim (N/C)$. On the other hand, since $\Gamma$ also acts properly and cocompactly on $X = \R^{n+2}$, its cohomological dimension must be $n+2$. This implies that  $\dim N/C =n+2$. 
Now, decompose $N$ as $N = \R \ltimes N_1$, where the $\R$-factor is generated by a one-parameter subgroup of $G$ with a nontrivial projection on $\R_{\bf L}$, and where  $$N_1 := N \cap (\R_{\bf H} \times K) \ltimes \Heis,$$ satisfies $\dim N_1 \geq n+1$. We have $p_{\bf H}(N_1) = \R$, so we further decompose $N_1$ as $N_1 = \R \ltimes N_0$, where the $\R$-factor is a one-parameter subgroup $s(t)$ of  $(\R_{\bf H} \times K) \ltimes \Heis$ that has a nontrivial projection on $\R_{\bf H}$, and $$N_0 := N_1 \cap (K \ltimes \Heis),$$ with $\dim N_0 \geq n$. 
Since $N_0$ is nilpotent, its projection to $K$ is an abelian  subgroup of $\O(n)$, whose maximal possible dimension is  $\lfloor \frac{n}{2} \rfloor$. Consequently, $N_0$ must intersect $\Heis$ non-trivially. Moreover, $N_1$ contains the subgroup  $s(t) \ltimes (N_0 \cap \Heis)$, which is therefore nilpotent. However, since the action of $s(t)$ on $N_0 \cap \Heis$ is semisimple (and unipotent), it must be trivial.  This is impossible since  $s(t)$ projects non-trivially on $\R_{\bf H}$, so its real eigenvalues are all different from $1$ (the real eigenvalues of its projection to $K$ are all equal to $1$).     
\end{proof}

\begin{proposition}\label{Proposition: p(Gamma)=R}
If $\overline{p(\Gamma)} = \R$, the action of $\Gamma$ is properly discontinuous and cocompact if and only if $\Gamma$ is contained in the isometry group. 
\end{proposition}
\begin{proof}
We have a short exact sequence $1 \to K \ltimes \Heis  \to G \to \R_{\bf L} \times \R_{\bf H} \to 1$. 
Let $\R_{\bf T}=\{e^{t T}, t \in \R \}$ be a one-parameter subgroup of $G$ such that $p(\R_{\bf T}) = \overline{p(\Gamma)}$. 
Then, $\Gamma$ is contained in $G_1:=\R_{\bf T} \ltimes (K\ltimes \Heis)$. 

We will show that this group is, in fact, isomorphic to a group of the form $(\R \times K) \ltimes \Heis$. 
Write ${\bf T}=\lambda+\Phi+\omega$, with $\lambda \in \R H \oplus \R L$, $\Phi \in \mathfrak{k}$, and $\omega \in \heis$. 
For any $\Psi \in \mathfrak{k}$, we have 
\begin{equation}\label{Eq_[T,Psi]}
[{\bf T},\Psi]=[\Phi,\Psi]+[\omega,\Psi] \in \mathfrak{k} \ltimes \heis. 
\end{equation} 
Since both $\ad_{\bf T}$ and $\ad_\Psi$  preserve the decomposition $\heis=\a \oplus \z$, it follows that $\ad_{[{\bf T},\Psi]}$, and hence, by (\ref{Eq_[T,Psi]}), also $\ad_{[\omega,\Psi]}$, preserve this decomposition. Thus, $[\omega,\Psi]$ must lie in the center $\z$ of $\heis$. Therefore, $\Ad_{\bf T}$ maps $\mathfrak{k}$ into $\mathfrak{k} \oplus \z$, implying that the $\R_{\bf T}$-action sends $K$ onto a compact subgroup of $K \times Z$. However, any compact subgroup of $K \times Z$ must be contained in $K$. This implies that the $\R_{\bf T}$-action on $K \ltimes \Heis$ preserves $K$.  By Lemma \ref{Lemma product R x K}, $G_1$ is isomorphic to a group of the given form.

Since $\R_{\bf T}$ projects non-trivially on $\R_{\bf L}$, the space $X$ can be identified with the quotient $(\R_{\bf T} \times K) \ltimes \Heis / K \ltimes A^+$. Moreover, $\Gamma$ projects as a dense subgroup of $\R_{\bf T}$. 
Thus, we are in the setting of  Lemma \ref{Lemma appendix}, which implies that $\Gamma$ is a cocompact lattice in a connected nilpotent subgroup $N$ of $G$.  Furthermore, since $N$ also acts cocompactly on $X=(\R_{\bf T} \times K) \ltimes \Heis / K \ltimes A^+$, it follows from \cite[Proposition 5.8]{article2_crelle} that $N$ contains the center of $\Heis$. Now, assume that $\R_{\bf T}$ has a non-trivial projection on $\R_{\bf H}$. 
Then $N$ contains two elements $x$ and $z$, where $x$ has a non-trivial projection on $\R_{\bf H}$, and $z$ belongs to the center of $\Heis$. The  actions of $\R_{\bf L}$ and $K$ on the center of $\Heis$ are trivial, but the $\R_{\bf H}$ action is non-trivial. 
This yields  $\Ad_x(z)=\lambda z$, with $\lambda \neq 1$, which contradicts the nilpotency of $N$. Therefore, $\R_{\bf T}$ must be contained in $\R_{\bf L} \times K$, which implies that $\Gamma$ is contained in the isometry group. 
\end{proof}

\begin{proposition}\label{Proposition: p(Gamma)=RxZ ou ZxZ}
If $\overline{p(\Gamma)} \simeq \R \times \Z$ or $\Z^2$, then the action of $\Gamma$ cannot be cocompact and properly discontinuous on $X$. 
\end{proposition}
\begin{proof}
By Corollary \ref{Corollary 1}, we have $p(\Gamma) \cap  \R_{\bf H} = \{1\}$. Under the assumption on $\Gamma$, this  implies that the projection of $\Gamma$ on $\R_{\bf L}$ is dense, 
which ensures the existence of an element $\gamma = (a,x) \in \Gamma$  such that the adjoint action $\Ad_a \in \Aut(\heis)$  has all its eigenvalues with norm $< 1$. 
Therefore, by Corollary \ref{Corollary: sequence converging to 1}, 
if $\Gamma \cap (K \ltimes \Heis) \neq \{1\}$, the action of $\Gamma$ is not properly discontinuous. We will show that this intersection is non-trivial. 
We first consider the case $\overline{p(\Gamma)} \simeq \R \times \Z$. Assume, for contradiction, that $\Gamma \cap (K \ltimes \Heis) = \{1\}$. Then $\Gamma$ is abelian. 
Let $\gamma_1=(b,y) \in \Gamma$. By Lemma \ref{Lemma fixed point: (A,b)-->(A,0)}, we can assume $x=1$. 
Since $\Gamma$ is abelian, we have $\gamma \gamma_1 \gamma^{-1}=(aba^{-1},a(y))= (b,y)$.  In particular, this implies that $a(y)=y$, and thus $y=1$. Consequently, $\Gamma$ is contained in $\R_{\bf H} \times \R_{\bf H} \times K$. Since $K$ is compact, the projection of $\Gamma$ on $\R_{\bf H} \times \R_{\bf L}$ must be discrete. Contradiction. Now, assume $\overline{p(\Gamma)} \simeq  \Z^2$. If $\Gamma \cap (K \ltimes \Heis) = \{1\}$, then $\Gamma$ is isomorphic to $\Z^2$, which implies that $\mathsf{cd}(\Gamma)=\mathsf{rank}(\Z^2)=2$. However, if $\Gamma$ acts properly discontinuously and cocompactly on $X$, then $\mathsf{cd}(\Gamma)=\dim X = n+2$, leading to a contradiction.  
\end{proof}

\begin{proposition}\label{Proposition: p(Gamma)=Z}
If $p(\Gamma) \simeq \Z$, then, up to finite index, $\Gamma$ is isomorphic to $\langle \hat{\gamma} \rangle \ltimes \Gamma_0$, with $\Gamma_0$ contained in $\Heis$. 
Moreover, if $\Gamma \not \subset \Isom(X)$, then $\Gamma_0$ is  abelian. 
\end{proposition}

\begin{proof}
Let $\Lambda := \Gamma \cap (K \ltimes \Heis)$. 
Since $p(\Gamma)$ is cyclic, the exact sequence $1 \to \Lambda \to \Gamma \to p(\Gamma) \to 1$ splits. Let $\hat{\gamma} \in \Gamma$ be an element that projects to a generator of $p(\Gamma)$.  Then, $\Gamma=\langle \hat{\gamma} \rangle \ltimes \Lambda$. Moreover, $\Lambda$ acts properly discontinuously and cocompactly on $K \ltimes \Heis / K \ltimes A^+$. 
By Lemma \ref{Lemma appendix}, there exists a finite index subgroup $\Lambda'$ of $\Lambda$ that is nilpotent. In particular, the quotient space $\Lambda' \backslash K \ltimes \Heis / K \ltimes A^+ $ is a compact manifold modeled on $(\Aff(\R^{n+1}), \R^{n+1})$, with nilpotent holonomy. By \cite[Theorem A]{fried1981affine},  a compact affine manifold with a nilpotent holonomy group is complete if and only if it has unipotent linear holonomy. This implies that $\Lambda'$ has a trivial projection on $K$, and hence is contained in $\Heis$. Thus, $\Lambda' \subset \Lambda \cap \Heis$ has finite index in $\Lambda$. Consequently, $\Lambda \cap \Heis$  also has finite index in $\Lambda$ and is normalized by $\Gamma$. 
Let $\Gamma_0 := \Gamma \cap  \Heis$. Then $\langle \hat{\gamma} \rangle \ltimes \Gamma_0$ has finite index in $\Gamma$. 
Now, if $\Gamma_0$ is non-abelian,  then 
it intersects $Z$ non-trivially. However, by Lemma \ref{Lemma: Gamma does not intersect the center}, this is not possible when $\Gamma$ is not contained in the isometry group. 
\end{proof}
\begin{proof}[Proof of Theorem \ref{Theorem 1}]
This follows from Propositions \ref{Proposition: p(Gamma)=RxR}, \ref{Proposition: p(Gamma)=R}, \ref{Proposition: p(Gamma)=RxZ ou ZxZ}, and \ref{Proposition: p(Gamma)=Z}.      
\end{proof}

\begin{remark}\label{Remark: proper action of Gamma}
Let $\Gamma$ be a group of the form $\Gamma= \langle \hat{\gamma} \rangle \ltimes \Gamma_0$, with $\Gamma_0 \subset \Heis$, and let $N_0$ be the Malcev closure of $\Gamma_0$ in $\Heis$. Then $\Gamma$ acts cocompactly and properly discontinuously on $X$ if and only if the following two conditions are satisfied:
\begin{itemize}
    \item $p_{\bf L}(\hat{\gamma}) \neq 1$, $\n_0 \cap \a^+ = \{0\}$ and $\dim \n=n+1$. 
    \item $e^{t L}(\n_0) \cap \a^+ = \{0\}$ for all $t \in \R$. 
\end{itemize}
\end{remark}\medskip

\subsection{Cases where $\Gamma$ is contained in the isometry group}
The following proposition contains in particular the case of Cahen-Wallach spaces of imaginary type, but also more general plane waves, including those with a unipotent $\R_{\bf L}$-action on $\Heis$. 

\begin{proposition}\label{proposition: cases of isometric actions}
If the semisimple part of the $\R_{\bf L}$-action on $\Heis$ is elliptic (i.e. the eigenvalues of ${\bf L} \in \Der(\heis)$ are purely imaginary), then a subgroup of $G$ acting properly discontinuously and cocompactly on $X$ is contained in the isometry group.
\end{proposition}
Note that if the condition in the above proposition holds, it holds for any ${\bf L} + \Phi+ \omega \in \Der(\heis)$, where $\Phi \in \k$ and $\omega \in \heis$. Therefore, it does not depend on the representative in $G$ of the $\R$-action on $\Heis$.   
\begin{proof}
By Theorem \ref{Theorem 1}, the only case to consider is when   $p(\Gamma)=\Z$. In this case, the proposition is a straightforward consequence of Corollary \ref{Corollary: sequence converging to 1}. 
\end{proof}

As mentioned above, Proposition \ref{proposition: cases of isometric actions} recovers \cite[Theorem 1.1]{Leis_Teiss}, which states that for Cahen-Wallach spaces of imaginary type, any discrete subgroup of  $G$ acting properly discontinuously and cocompactly on $X$ is contained in the isometry group.  The reason is that being of imaginary type implies that the $\R_{\bf L}$-action on $\Heis$, which is semisimple for any Cahen-Wallach space, is elliptic. This can be seen by looking at the Heisenberg action on $X$, as  explained below. 
Consider a Cahen-Wallach space 
\begin{equation}\label{Eq: CW metric}
X=(\R^{n+2}, 2dudv + x^\top S\, x\, du^2 + \sum_{i=1}^{n} dx_i^2).  
\end{equation}
The parallel null vector field $V$ is given by $\partial_v$ and the $\F$-foliation, tangent to $V^\perp$,  is defined by the hyperplanes $\R^{n+1} \times \{u\}$, for $u \in \R$.  
The Heisenberg group acts by preserving individually the leaves of $\F$.

Fix the leaf $\{u=0\}$, on which $\Heis$ acts via a specific representation $\pi_0: \Heis \to \Aff(\R^{n+1})$. 
The action on the leaves $\{u=t\}_{t \in \R}$  is then determined by the $\R_{\bf L}$-action on $\Heis$ and $\pi_0$. To describe this action, define for $t \in \R$, the map $\tau^t: (y=(v,x),u) \mapsto (y,u+t)$. This forms a one-parameter group of isometries of $X$, acting transversely to the foliation $\F$. Consider the conjugacy morphism 
\begin{align*}
\mathfrak{P}_t: \Heis &\to  K \ltimes \Heis \\
h &\mapsto \tau^{-t} h \tau^{t}
\end{align*}
By \cite[Lemma 2.4]{article1}, we have $\mathfrak{P}_t(\Heis)=\Heis$ for every $t \in \R$, defining a one-parameter subgroup $\mathfrak{P}: \R \to \Aut(\Heis), t \mapsto \mathfrak{P}_t,$  of $\Aut(\Heis)$. 
Let $h \in \Heis$ and $(y,t) \in \R^{n+1} \times \R$. Then $h$ acts on $(y,t)$ as follows:
$$h \cdot (y,t)=h \circ \tau^t(y,0)=\tau^t \circ \tau^{-t} h \tau^t(y,0)=\tau^t \circ \mathfrak{P}_t(h) (y,0) = (\pi_0 \circ \mathfrak{P}_t(h) (y),t).$$  

This allows us to write the action of $\Heis$   on $X=\R^{n+1} \times \R$ as follows
		\begin{align*}
			 \Heis \times \R^{n+1} \times \R &\to  \R^{n+1} \times \R\\
			(h, (y, u) ) & \mapsto (\pi_0 \circ \mathfrak{P}_{u}(h) (y), u),
		\end{align*} 
		where  $\mathfrak{P}: \R \to \Aut(\Heis)$ is the one-parameter subgroup of $\Aut(\Heis)$ above. It is given by $e^{u {\bf L_1}}$, for some ${\bf L_1}={\bf L} + \Phi + \omega$, $\Phi \in \k, \;\omega \in \heis$. \medskip

On the other hand, the action of $\Heis$ can also be described explicitly in the coordinates $(v,x,u)$. We write it here, and refer to \cite[Section 2]{article1} or \cite[Section 3.2]{Leis_Teiss} for the proof. In the representation $\pi_0$, the center $Z=\R$ acts on $\R^{n+1}$ by translation along the $v$-coordinate, and $(\lambda', \lambda) \in \R^{2n}$, a complement of $Z$, acts through the unipotent affine transformation 	
	$$ \left(\begin{pmatrix}
		1 & \lambda'   \\
		0 & I_n   \\
	\end{pmatrix},
	\begin{pmatrix}
		\frac{ \langle\lambda, \lambda' \rangle}{2}   \\
		\lambda   \\
	\end{pmatrix}\right) 
	\in \SL_{n+1} (\R) \ltimes \R^{n+1}.$$ 
Thus, an element $h=(\lambda',\lambda,c) \in \Heis$, with $(\lambda',\lambda) \in \R^{2n}$ and $ c \in Z$,  acts on $X=\R^{n+1} \times \R$ as $$h \cdot (y,u) = (\pi_0((\beta'(u),\beta(u),c)), u).$$ where  $\beta$ is the solution to the differential equation $\beta''(u)=  S \beta(u)$, with initial conditions $(\lambda',\lambda)$ at $u=0$.     
Now, as explained in \cite[p. 3]{Leis_Teiss}, the imaginary type condition ensures the boundedness of the solutions $\beta(u)$. 
From the previous discussion, it follows that the boundedness of all these solutions is equivalent to the condition that, for $h \in \Heis$, the curves $u \in \R \mapsto \mathfrak{P}_{u}(h) \in \Heis$ are all bounded with respect to $u$. This occurs precisely when the $\R_{\bf L}$-action on $\Heis$ is elliptic.  \\

\subsection{Proper and cocompact similarity non-isometric actions} 
In this paragraph, we construct properly discontinuous and cocompact actions by similarities on homogeneous plane waves, where the action is not isometric.

\begin{example}[An example in dim $4$]\label{Example 1}
Fix a basis $\mathcal{B}=(e_1,e_2,e_3)$ of $\R^3$. Consider the Lie group $P := \R_{\bf L_1} \ltimes \R^3$, where $\R_{\bf L_1}$ acts on $\R^3$ via the one-parameter group $e^{t {\bf L_1}}$, with ${\bf L_1} \in \Der(\R^3)$ given in the basis $\mathcal{B}$ by
$$
{\bf L_1} :=\begin{pmatrix}
- 1 & 0 & 0 \\
0 & - 3 & 0 \\
0 & 0 & 0
\end{pmatrix}.
$$
We define a left-invariant Lorentzian metric $g$ on $P$ such that the induced metric on $\R^3$ is degenerate, with its null direction field collinear to  the left-invariant vector field associated to $e_3$, the fixed point of $e^{t {\bf L_1}}$.  By  \cite[Theorem 3]{Leis}, this induced vector field is parallel (and null), and the space $(P,g)$ is a plane wave. \\
Now, consider the one-parameter subgroup $h(t):=e^{t {\bf H_1}} \subset \Aut(\R^3)$, where $$
{\bf H_1}:=
\begin{pmatrix}
     1 & 0 & 0\\
     0 & 1 & 0 \\
     0 & 0 & 2
\end{pmatrix}.
$$  Since the actions of $\R_{\bf L_1}$ and $\R_{\bf H_1}$ on $\R^3$ commute, $h(t)$ extends naturally to a one-parameter subgroup of $\Aut(P)$, acting trivially on $\R_{\bf L_1}$. Moreover, its action on $(P,g)$ is by homotheties, fixing the point $(1, 0_{\R^3})$. As a result, $(P,g)$  can be identified with the quotient $(\R_{\bf H_1} \times \R_{\bf L_1}) \ltimes \R^3/ \R_{\bf H_1}$.

Define $
{\bf T}:= {\bf L_1} + {\bf H_1} 
$. Consider the Lie subgroup $P_{\bf T}= \R_{\bf T} \ltimes \R^3 \subset (\R_{\bf H_1} \times \R_{\bf L_1}) \ltimes \R^3$, where the action on $\R^3$ is given by $e^{t {\bf T}}$. This subgroup and all its conjugates are transversal to the isotropy and thus acts simply transitively on $P$, preserving its conformal structure. Consequently, $P_{\bf T}$ inherits a left-invariant conformal structure isomorphic to that of $P$. Moreover, the element $x:=e^{\bf T} \in P_{\bf T}$ preserves a lattice $\Gamma_0$ in $P_{\bf T} \cap \R^3$, and the group $\Gamma:= \langle x \rangle \ltimes \Gamma_0$ forms a cocompact lattice in $P_{\bf T}$. As a result, the quotient space $\Gamma \backslash P_{\bf T}$ is compact and locally conformal to $P$, but it does not globally inherit the plane wave metric of $P$. \medskip  
\end{example}
 
In the example above, the plane wave metric on  $P$ was not explicitly specified. However, such an example can be realized using a non-conformally flat Cahen-Wallach space (note that in dimension $3$, all plane waves are conformally flat).
Let $(X,V)$ be a $1$-connected, non-flat, complete homogeneous plane wave. By \cite[Theorem 1.5]{article1}, the connected component of $\Isom(X)$ is given by $G_\rho=(\R_{\bf L} \times K) \ltimes \Heis$,  where the $\R_{\bf L}$-action is defined by $\rho(t)=e^{t {\bf L}}$, ${\bf L} \in \Der(\heis)$, with 
          $${\bf L}=\begin{pmatrix}
           F & B & 0\\
           I & F & 0\\
           0 & 0 & 0
        \end{pmatrix}$$
    written in the decomposition  $\heis=\a^+\oplus \a^- \oplus \z$, with $F$ antisymmetric and $B$ symmetric. 
In dim $4$,  we have $F:=
\begin{pmatrix}
    0 & -a \\
    a & 0
\end{pmatrix}
$ and $B=
\begin{pmatrix}
    b & c\\
    c & d
\end{pmatrix}
$. 
The metric on $X$ defines a Cahen-Wallach space if and only if $[F,B]=0$ and $B$ is non-degenerate \cite[Remark 5.14]{article1}. In this case, \cite[Theorem 6.3]{article1} gives $S=B$ in (\ref{Eq: CW metric}). By \cite[p. 15]{blau_lecture} or \cite[Proposition 3.1]{Leis_Teiss},  
it is conformally flat if and only if $B=\lambda I$ for some $\lambda \in \R$. 
So, for $(a,b,c,d)=(0,6, \sqrt{15},4)$, we obtain a non-conformally flat Cahen-Wallach space. The corresponding matrix ${\bf L}$ is conjugate to the diagonal form
$$\mathrm{Diag}(-3,-1,1,3,0),$$
in the decomposition $\heis_5= \a \oplus \z$. Next, consider the (abelian) ideal $\n$ of $\heis_5$ generated by the center and the eigendirections associated with the eigenvalues $-1$ and $-3$.  Then, $\n$ is preserved by $e^{t {\bf L}}$. Furthermore, the group $P$ from Example \ref{Example 1} embeds in $G_\rho$ as $\R_{\bf L} \ltimes N$, where $N$ is the subgroup of $\Heis_5$ with Lie algebra $\n$.  This group and all its conjugates are transversal to the isotropy. 
Therefore, $P$ embeds as a simply transitive subgroup of $G_\rho$, and thus inherits a left-invariant Cahen-Wallach metric which is isometric to that of $X$.\bigskip

\textbf{Dimension 3.} It is known from \cite{KO, DZ, allout2022homogeneous} that $3$-dimensional Lorentzian compact manifolds locally isometric to Cahen-Wallach spaces do not exist. However, as we will see in Example \ref{Example 2},  conformal compact quotients exist precisely when the $\R_{\bf L}$-action on the Heisenberg group is hyperbolic. 

\begin{example}[Conformal compact quotients of $3$-dim Cahen-Wallach spaces]\label{Example 2}
Let $X$ be a $1$-connected, $3$-dimensional Cahen-Wallach space. From \cite[Theorem 5.13, Remark 5.14]{article1}, we know that the connected component of $\Isom(X)$ is $G_\rho=\R_{\bf L} \ltimes \Heis_3$,  with the $\R_{\bf L}$-action given by $\rho(t)=e^{t {\bf L}}$, ${\bf L} \in \Der(\heis_3)$, such that 
          $${\bf L}=\begin{pmatrix}
           0 & b & 0\\
           1 & 0 & 0\\
           0 & 0 & 0
        \end{pmatrix},$$
written in the decomposition  $\heis_3=\a^+\oplus \a^- \oplus \z$, with $b \in \R \smallsetminus \{0\}$.  
The connected component of the conformal group is $(\R_{\bf H} \times \R_{\bf L})\ltimes \Heis_3$, where 
 $${\bf H}=\begin{pmatrix}
           1 & 0 & 0\\
           0 & 1 & 0\\
           0 & 0 & 2 
        \end{pmatrix}.$$
The space $X$ is a hyperbolic (resp. elliptic) Cahen-Wallach space, meaning that the $\R_{\bf L}$-action on $\Heis_3$ is hyperbolic (resp. elliptic), when $b>0$ (resp. $b<0$). In the elliptic case, we know from Proposition \ref{proposition: cases of isometric actions} that for any conformal compact quotient $\Gamma \backslash X$,  the group $\Gamma$ is contained in the isometry group. However, from \cite{KO, DZ, allout2022homogeneous}, we know that such quotients do not exist. 

We will construct a conformal compact quotient in the hyperbolic case, i.e. for any $b>0$. In this case, the non-zero eigenvalues of ${\bf L}$ are $\pm\sqrt{b}$, with corresponding eigenvectors $X \pm b Y$.  Let $Y_1$ be the eigenvector corresponding to the eigenvalue $-\sqrt{b}$, and define ${\bf T}:=\alpha {\bf L}+ {\bf H}$, $\alpha >0$. The (abelian) ideal $\n \subset \heis_3$, generated by the center of $\heis_3$ and $Y_1$, is preserved by the action of $e^{t {\bf T}}$. Consider the group $P_{\bf T}:=\R_{\bf T} \ltimes N$, where $\R_{\bf T}$ acts on $N$ through the restriction of $e^{t {\bf T}}$ to $N$.  Then,  $P_{\bf T}$ and all its conjugates are transversal to the isotropy, so $P_{\bf T}$ acts simply transitively on $X$, preserving the conformal structure. Consequently, $P_{\bf T}$ inherits a left-invariant conformal structure isomorphic to that of $X$. \\
Moreover, for $\alpha$ such that $-\alpha \sqrt{b} = -3$, 
the restriction of $e^{\bf T}$ to $N$ is conjugate to the hyperbolic matrix $\mathrm{Diag}(e^{2}, e^{-2})$, hence preserves a lattice $\Gamma_0$ in $N$. Therefore, $\Gamma:= \langle x \rangle \ltimes \Gamma_0$, with $x:=e^{\bf T}$, forms a cocompact lattice in $P_{\bf T}$, and the quotient space $\Gamma \backslash P_{\bf T}$ is compact and locally conformal to $X$, but does not inherit  the Cahen-Wallach metric of $X$.
\end{example}

\section{Non-essentiality of the conformal groups}
The aim of this section is to prove the following result

\begin{theorem}\label{Theorem: non-essentiality}
Consider a compact conformal quotient $M=\Gamma \backslash X$.
The conformal group of $M$ is non-essential. More precisely, 
\begin{itemize}
    \item[i)] either $\Gamma$ is contained in the isometry group of $X$, in which case $\Gamma \backslash X$ has a plane wave metric coming from that of $X$,
    \item[ii)] or, it is not contained  in the isometry group of $X$, and in this case,   
    the conformal group preserves a Lorentzian metric on $M$ conformal to the plane wave metric on $X$. %This metric is not a plane wave metric. 
\end{itemize}   
\end{theorem}

Let $M=\Gamma \backslash X$ be a (compact) conformal quotient of $X$. The conformal group of $M$ is given by $N_{\widehat{G}}(\Gamma)$, the normalizer of $\Gamma$ in the conformal group $\widehat{G}$. 
The following remark states that to prove the non-essentiality of the conformal group, it is sufficient to prove it on a normal covering space. Therefore, there is no harm in replacing $\Gamma$ by the finite index subgroups obtained in Theorem \ref{Theorem: non-essentiality}.  And, as in the previous section, we will assume that  $\Gamma \subset G$. 

\begin{remark}\label{Remark: replace by a cover}
Let $X$ be a $1$-connected conformal Lorentzian space, and let $M=\Gamma \backslash X$  be a conformal quotient, where $\Gamma$ is a discrete subgroup of $G:=\Conf(X)$. Consider a covering space  $M_0=\Gamma_0 \backslash X$ of $M$, where $\Gamma_0$ is a normal subgroup of $\Gamma$. If $\Conf(M_0)$ is non-essential on $M_0$, then $\Conf(M)$ is non-essential on $M$. Indeed, $\Conf(M_0)=N_G(\Gamma_0)$ contains $\Gamma$, so a Lorentzian metric in the conformal class,  preserved by $\Conf(M_0)$, is $\Gamma$-invariant. Thus, it descends to a Lorentzian metric on $M$, which is preserved by $\Conf(M)$, as $\Conf(M)$ is contained in $\Conf(M_0)$. This is in particular the case when $\Gamma_0=\{0\}$, and $M_0=X$ is the universal cover of $M$. 
\end{remark}

\begin{observation}\label{observation: Gamma contained in Isom(X)}
Consider a compact conformal quotient $M=\Gamma \backslash X$. When $\Gamma$ is contained in the isometry group of $X$, the compact quotient $M:=\Gamma \backslash X$ inherits an induced Lorentzian metric, and the conformal group of $M$ acts by homotheties. However, since $M$ is compact, it does not admit any pure homothetic transformation (see \cite[Corollary 2.1]{Alekseevski_selfsimilar}). Consequently, the conformal group of $M$ coincides with its isometry group.     
\end{observation}

By Observation \ref{observation: Gamma contained in Isom(X)}, the only case to see is when $\Gamma \not \subset \Isom(M)$. This will be the subject of the next paragraph. 

\subsection{Normalizer of the fundamental group} This paragraph aims to prove the following proposition, which gives a description of the normalizer of the fundamental group, and which is crucial for the proof of Theorem \ref{Theorem: non-essentiality}.

\begin{proposition}\label{Proposition: N_G(Gamma)}
Consider a compact conformal quotient $M=\Gamma \backslash X$. 
\begin{enumerate}
    \item  If $\Gamma \subset \Isom(X)$, then  $N_G(\Gamma) \subset \Isom(X)$. 
    \item   If $\Gamma \not \subset \Isom(X)$, then $N_G(\Gamma)$ is a cyclic extension of a subgroup of  $K \ltimes \Heis$.   
\end{enumerate}  
\end{proposition}

In the proposition above, we are considering $N_G(\Gamma)$ instead of the normalizer in the full conformal group. This is justified by the lemma below.

\begin{lemma}\label{Lemma: normalizer contained in G}
Let $\Gamma \backslash X$ be a compact conformal quotient. If $\Gamma\not\subset \Isom(X)$, then the normalizer of $\Gamma$ in the full conformal group $\widehat{G}$ is contained in $G$. 
\end{lemma}

\begin{proof}
Assume that there exists an element $g \in N_{\widehat{G}}(\Gamma)$ whose projection in $E_{\bf L}(1) \times \R_{\bf H}$ has a non-trivial projection on $\langle \sigma \rangle$. Up to composing by an element in $\Gamma$ with non-trivial $\R_{\bf H}$-component, we may assume that $g$ has also a non-trivial projection on $\R_{\bf H}$. Then $g^2$ has a trivial projection on $E_{\bf L}(1)$, and a non-trivial projection on $\R_{\bf H}$. Consequently, either $\Ad_{g^2}$ or $\Ad_{g^{-2}} \in \Aut(\heis)$ has all its eigenvalues with norm  $< 1$. By Corollary \ref{Corollary: sequence converging to 1}, this yields $\Gamma \cap \Heis = \{1\}$. This contradicts the description of $\Gamma$ obtained in Theorem \ref{Theorem 1}.         
\end{proof}

To prove Proposition \ref{Proposition: N_G(Gamma)}, we start with the following lemma.

\begin{lemma}\label{lemma: description of projection of N_G(Gamma)}
Let $\Gamma_0 \subset G$ be a discrete subgroup contained in $ \Heis$, and let $N_0$ be its Malcev closure in $\Heis$. Consider $Q:=p(N_G(\Gamma_0))$, which is an abelian subgroup of $\R_{\bf H} \times \R_{\bf L} \simeq \R^2$.
\begin{enumerate}
    \item If $Q$ is discrete, then either $Q = \{1\}$ or $Q \simeq \Z$. 
    \item If $Q$ is non-discrete, then $\overline{Q}\simeq \R$. In this case, if $Z \subset N_0$, then $N_G(\Gamma_0) \subset \Isom(X)$. 
\end{enumerate}
\end{lemma}

The following fact is needed in the proof. 
\begin{fact}\cite[Theorem 4.1.6]{Thurston}\label{Thurston: strongly Zassenhaus neighborhood}
Let $H$ be a Lie group. There exists a neighborhood $U$ of $1$ in $H$ such that any discrete subgroup $\Gamma$ of $H$ generated by $U \cap \Gamma$ is nilpotent. 
We call such an identity neighborhood $U$ a \textit{Zassenhaus neighborhood}.
\end{fact}
In fact, \cite[Theorem 4.1.7]{Thurston} gives an even stronger statement: there exists a neighborhood $U$ of $1$ in $H$ such that any discrete subgroup $\Gamma$ generated by $U \cap \Gamma$ is a lattice in a  connected nilpotent subgroup of $H$. Here, we only need the weak version of this statement.

\begin{proof} 
If $\overline{Q} \simeq \R^2, \Z \times \R$, or $\Z^2$, there exists $g \in N_G(\Gamma)$ such that $\Ad_g \in \Aut(\heis)$ has all its eigenvalues with norm $< 1$. This is not possible by Corollary \ref{Corollary: sequence converging to 1}.

Assume now that $\overline{Q}  \simeq \R$. We will show that for any neighborhood $U$ of the identity in $\R_{\bf H} \times \R_{\bf L} \times K$, there exists $g \in N_G(\Gamma_0)$ such that $r(g) \in U$.
The projection $p(N_G(\Gamma_0))$  is dense in a $1$-dimensional subgroup  $D$ of $\R_{\bf H} \times \R_{\bf L}$, that has a non-trivial projection on $\R_{\bf H}$. Let $l: \R \to G$ be a one-parameter subgroup of $G$ that projects onto $D$. Then $r(N_G(\Gamma_0)) \subset l \ltimes K$. 
Consider the projection $q: l \ltimes K\to  l$, and let $F:=\overline {r(N_G(\Gamma_0))}$.  Since $K$ is compact,  $q(F)$ is closed in $l$, hence equal to  $\overline {p(N_G(\Gamma_0))}=l$. Therefore, the restriction of the projection $q$ to $F$ is a Lie group homomorphism $F \to l$, which is surjective. This defines a fiber bundle 
$ F\cap K \hookrightarrow F\to l$.
Consequently, the identity component $F^o$ also projects surjectively onto $l$.
Since $r(N_G(\Gamma_0)) \cap F^o$ is dense in $F^o$,  it intersects any neighborhood $U$ of $1 \in l \ltimes K$. 

Now, let $U_1$ be a neighborhood of $1$ in $\R_{\bf H} \times \R_{\bf L} \times K$, and let $U_2$ be a neighborhood of $1$ in $\Heis$, such that $U_1 \times U_2$ is a Zassenhauss neighborhood. Consider an element  $g=(a,x) \in N_G(\Gamma_0)$ with $a \in U_1$. Since $g$ normalizes $\Gamma_0$, we define the subgroup $$\tilde{\Gamma} = \langle g \rangle \ltimes \Gamma_0.$$ This is a discrete group generated by $g$ together with a finite set of generators of $\Gamma_0$. 
Thus, we can write a generating set for $\tilde{\Gamma}$ as $\{\gamma_i=(a_i,x_i)\}_{i=1}^m$, where $a_i \in U_1$ and $x_i\in\Heis$.
Let $A$ be a $K$-invariant complement of the center in $\Heis$. 
We consider the automorphism $\Psi$ of $G$ which is the identity on $\R_{\bf H} \times \R_{\bf L} \times K$, multiplication by $\lambda \in\R_{>0}$ on $A$ and multiplication by $\lambda^2$ on the center of $\Heis$. Choosing $\lambda$ sufficiently small ensures that $\Psi(x_i)\in U_2$ for all $i=1,\dots,m$. Consequently,  $\Psi(\tilde{\Gamma})=\langle (a_i,\Psi(x_i)),\ i=1,\dots,m\rangle$ is generated by elements of $U_1\times U_2$ and, by Lemma \ref{Lemma appendix},  is nilpotent. Hence,  $\tilde{\Gamma}$ is also nilpotent.

Now, assume that $Z \subset N_0$. Then, $a(N_0) \subset N_0$, and the restriction $a_{\vert N_0}$ is unimodular. 
Indeed, for any element $\gamma=(1,y) \in \Gamma_0$, we have 
$g \gamma g^{-1}=  x a(y) x^{-1} \in \Gamma_0$.    
Hence $a(y) \in \Gamma_0 Z \subset N_0$. Since $\Gamma_0$ generates $N_0$, we obtain $a(N_0) \subset N_0$. Furthermore, the adjoint action $\Ad_g=\Ad_x \circ \Ad_a$ preserves a lattice $\Gamma_0$ in $N_0$, so $\Ad_{g_{\vert N_0}}$ is unimodular. Since $\Ad_{x_{\vert N_0}}$ is unimodular, the claim follows.    

Suppose that $N_G(\Gamma_0) \not \subset \Isom(X)$. Then, $g=(a,x)$ projects non-trivially on $\R_{\bf H}$. By Lemma \ref{Lemma: Gamma does not intersect the center}, $\Gamma_0 \cap Z = \{1\}$, implying that $\Gamma_0$ is abelian. 
Let $\gamma \in \Gamma_0$, and define $\pi_1(\gamma):=[g,\gamma] \in \Gamma_0$. Define the sequence $\pi_k(\gamma):=[g,\pi_{k-1}(\gamma)]$. Since $\tilde{\Gamma}$ is nilpotent, there exists some $k \in \N$ such that for any $\gamma \in \Gamma_0$, $\pi_k(\gamma)=0$. On the other hand, we have  $\pi_1(\gamma)-(a-I)(\gamma) \in Z$, which yields $\pi_k(\gamma)-(a-I)^k \in Z$. Thus, $(a-I)^k(\Gamma_0) \subset Z$, and consequently, $(a-I)^k(N_0) \subset Z$. As a result, the restriction $a_{\vert N_0}$ has $1$ as an eigenvalue with multiplicity $\dim N_0 -1$. Moreover, for any $z \in Z$, we have $a(z)=\alpha z$, with $\alpha \in \R$. Since $a_{\vert N_0}$ is unimodular, it follows that $\alpha =1$, contradicting the fact that $g$ projects non-trivially on $\R_{\bf H}$. Therefore, $a$, and hence $Q$, has a trivial projection on $\R_{\bf H}$. 
\end{proof}

\begin{proof}[Proof of Proposition \ref{Proposition: N_G(Gamma)}]
The first statement follows from Observation \ref{observation: Gamma contained in Isom(X)}. 
Now, assume that $\Gamma \not \subset \Isom(X)$. Then, by Theorem \ref{Theorem 1}, we have, up to finite index, $\Gamma = \langle \hat{\gamma} \rangle \ltimes \Gamma_0$,  where $\Gamma_0 \subset \Heis$ is abelian, and $p_{\bf H}(\hat{\gamma})\neq 1$. Since $\Gamma_0$ is normal in $\Gamma$, we have $N_G(\Gamma) \subset N_G(\Gamma_0)$. Hence, by Lemma \ref{lemma: description of projection of N_G(Gamma)}, $\overline{p(N_G(\Gamma)}$ is either trivial, or isomorphic to $\Z$, or to $\R$. 
Since $\Gamma$ has a non-trivial projection on $\R_{\bf H}$, the same holds for  $p(N_G(\Gamma)$, ruling out the trivial case. Moreover, by Lemma \ref{lemma: description of projection of N_G(Gamma)}, this also excludes the possibility that $\overline{p(N_G(\Gamma)}$ is  isomorphic to $\R$. We conclude that the projection must be a cyclic group.   
\end{proof}

\subsection{Proof of Theorem \ref{Theorem: non-essentiality}}

We are now able to prove Theorem \ref{Theorem: non-essentiality}. We denote by $g$ the plane wave metric on $X$. We start with the following lemma:

\begin{lemma}\label{Lemma: construction of conformal metric}
Let $\gamma \in G$ be an element with a non-trivial projection on $\R_{\bf L}$. Then, there exists a function  $f \in C^\infty(X,\R)$ such that
\begin{enumerate}
\item $f$ is $\F$-constant. In particular, any $\phi \in K \ltimes \Heis$ acts isometrically on the conformally rescaled metric $e^f g$. 
    \item $\gamma$ is an isometry of  $e^f g$.
\end{enumerate}
\end{lemma}
\begin{proof}
If $\gamma \in \Isom(X)$, set $f \equiv 0$. Assume now that $\gamma$ acts essentially on $X$, i.e. its projection on $\R_{\bf H}$ is non-trivial. Write $\gamma^* g = e^\alpha g$, $\alpha \in \R \smallsetminus \{0\}$. One checks that $\gamma$ is isometric for $e^f g$ if and only if 
\begin{equation}\label{Eq f}
f \circ \gamma = f - \alpha.    
\end{equation}
The space of $\F$-leaves  is identified with $\xi=\R$, endowed with its $G$-invariant translation structure.   Let $u \in C^\infty(X,\R)$ be a ($G$-invariant) affine parameter on it. The invariance under  $\gamma$ implies $u \circ \gamma = u + b$, for some $b \in \R$.  
Since $\gamma$ has a non-trivial projection on $\R_{\bf L}$, we conclude that $b \in \R \smallsetminus \{0\}$. 
Denote by $\tau_b: u \mapsto u+b$ the (non-trivial) translation in the affine parameter induced by $\gamma$. 
Let $\overline{f} \in C^\infty(\R,\R)$, and set $f := \overline{f} \circ u$. Then $f$ satisfies  (\ref{Eq f}) if and only if 
\begin{equation}\label{Eq bar(f)}
\overline{f} \circ \tau_b = \overline{f} - \alpha.      
\end{equation}
Define a function $\overline{f} \in C^\infty(\R,\R)$ as follows:
\begin{itemize}
    \item Take any smooth function $\Phi_0$ on $[0,\epsilon[$, where $0<\epsilon <b$. Set $$\Phi_1 := \Phi_0 \circ \tau_b^{-1} - \alpha$$ on $[b, b + \epsilon[$, and let $\overline{f}_0$ be a smooth function on $I_0=[0, b]$ interpolating between $\Phi_0$ and $\Phi_1$. 
    \item Extend $\overline{f}$ to $I_{k}=[b+k-1,b+k[$, $k \geq 1$,  by $$\overline{f}_{\vert I_{k}} := \overline{f}_0 \circ \tau_b^{-k} -k \alpha .$$ 
\end{itemize}
Clearly, $\overline{f}$ is a well-defined smooth function on $\R$, satisfying (\ref{Eq bar(f)}). Thus, the function $f$ defined by $f = \overline{f} \circ u$ is a smooth function on $X$, that satisfies the required conditions (1) and (2). 
\end{proof}

\begin{proof}[Proof of Theorem \ref{Theorem: non-essentiality}]
When $\Gamma$ is contained in the isometry group of $X$, the conformal group of the compact quotient $M:=\Gamma \backslash X$ is non-essential by Observation $\ref{observation: Gamma contained in Isom(X)}$. 
Now, assume that $\Gamma$ is not contained in the isometry group of $X$. 
By Lemma \ref{Lemma: normalizer contained in G}, the normalizer of $\Gamma$ in $\widehat{G}$ is contained in $G$. So it is sufficient to consider $N_G(\Gamma)$. By Proposition \ref{Proposition: N_G(Gamma)}, 
we have $p(N_G(\Gamma)) \simeq \Z$. So,  $N_G(\Gamma) = \langle h \rangle \ltimes N_{K \ltimes \Heis}(\Gamma)$, where $h$ projects non-trivially on both $\R_{\bf L}$ and $\R_{\bf H}$. 
By  Lemma \ref{Lemma: construction of conformal metric}, there exists a function   $f \in C^\infty(X,\R)$ such that $h$ is an isometry for the metric $e^f g$. Moreover, since $N_{K \ltimes \Heis}(\Gamma)$ is contained   in $K \ltimes \Heis$, it also preserves $e^f g$. Thus, $N_G(\Gamma)$ preserves the Lorentzian metric  $e^f g$, which is then well defined on the quotient. Consequently, the conformal group of  $\Gamma \backslash X$ preserves the Lorentzian metric induced by $e^f g$, proving that the conformal group of  $\Gamma \backslash X$ is non essential. 
\end{proof}

\begin{proof}[Proof of Corollary \ref{Corollary: Non essentiality in the homog case}]
If a compact quotient of $X=G/I$ is essential, then, by Remark \ref{Remark: replace by a cover}, the universal cover $X=G/I$ is also essential. If it is not conformally flat, then by \cite{Alekseevsky2024}, it is conformal to a homogeneous plane wave.  The conclusion then follows from Theorem \ref{Theorem: non-essentiality}.  
\end{proof}

\bibliographystyle{abbrv}
\bibliography{Bibliographie.bib}\blfootnote{
The  author is  supported by the grants, PID2020-116126GB-I00\\
(MCIN/ AEI/10.13039/501100011033), and the framework IMAG/ Maria de Maeztu,\\ 
CEX2020-001105-MCIN/ AEI/ 10.13039/501100011033.\medskip}

\end{document}